\documentclass[12pt,reqno]{amsart}
\usepackage{a4wide,amsfonts,amsmath,latexsym,amssymb,euscript,eufrak,graphicx,units,mathrsfs}
\usepackage[utf8]{inputenc}
\usepackage{amsmath}
\usepackage{amsfonts}
\usepackage{amssymb}
\usepackage{amsthm}
\usepackage{floatrow}
\usepackage{blindtext}
\usepackage{multicol}
\usepackage[english]{babel}
\usepackage{enumerate}
\usepackage{eufrak}
\usepackage{graphicx}
\usepackage{caption}
\usepackage{subcaption}
\usepackage{float}
\usepackage{epstopdf}
\usepackage{multirow}
\usepackage{mathtools}
\usepackage{multirow}
\usepackage[usenames,dvipsnames]{color}

\numberwithin{equation}{section}

\newcommand{\HC}{\mathcal{H}}

\newcommand{\stepid}[1]{\mathbf{1}_{[0, #1]}}

\newcommand{\norm}[1]{\left\lVert#1\right\rVert}
\newcommand{\inner}[2]{\langle #1, #2 \rangle}

\newtheorem{theorem}{Theorem}[section]
\newtheorem{lemma}[theorem]{Lemma}

\theoremstyle{definition}

\theoremstyle{remark}
\newtheorem{remark}[theorem]{Remark}

\numberwithin{equation}{section} 

\begin{document}
\title[New Kolmogorov bounds in the CLT for   random ratios and applications]
{New Kolmogorov bounds in the CLT for   random ratios and applications}

\author[K. Es-Sebaiy]{ Khalifa Es-Sebaiy$^*$}
\address{Khalifa Es-Sebaiy ($^*$Corresponding author): Department of Mathematics, Faculty of Science, Kuwait University, Kuwait}
\email{khalifa.essebaiy@ku.edu.kw}
\author[F. Alazemi]{Fares Alazemi}
\address{Fares Alazemi: Department of Mathematics, Faculty of Science, Kuwait University, Kuwait}
\email{fares.alazemi@ku.edu.kw}

\begin{abstract}
\medskip
We develop techniques for determining an explicit Berry-Esseen bound
in the Kolmogorov distance for the normal approximation of  a ratio
 of   Gaussian functionals. We provide  an upper bound  in terms of the third and fourth
 cumulants, using  some novel techniques and   sharp estimates for cumulants.
 As applications, we study the rate of convergence of the distribution of   discretized  versions of minimum contrast and
maximum likelihood estimators of the drift parameter   of the
Ornstein-Uhlenbeck process. Moreover,  we derive upper bounds
  that   are strictly sharper than those available in the literature.\\

\end{abstract}

\maketitle



\allowdisplaybreaks

\renewcommand{\thefootnote}{\arabic{footnote}}

 \section{Introduction}
 The present work deals with a quantitative central limit
theorem (a Berry-Esseen bound) using the Kolmogorov distance for a
ratio of Gaussian functionals. An example of such a ratio is the
maximum likelihood statistics for
 parameter estimation in  linear stochastic differential  equations.
This is both well-motivated by practical needs and theoretically
challenging. In the parameter estimation context, our  motivation is
that  the study of the asymptotic distribution of parameter
statistics is not very useful in general for practical purposes
unless the rate of convergence is known.

The
 well-known key step  to study upper bounds of the Kolmogorov distance
for the rates of convergence of the distribution of the estimators
having the form of a ratio of two random variables such that the
numerator has the rate of convergence of its distribution  and the
denominator  is positive and  converges in probability to one is to
separate the numerator and the denominator through the following
technical result of Michel and Pfanzagl \cite[Lemma 1]{MP}: let $X$
and $Z$ be any tow random variables on a probability space $(\Omega,
\mathcal{F}, \mathbb{P})$ with $P(Z>0)>0$. Then, for any
$\varepsilon>0$, we have 
\begin{eqnarray}d_{Kol}\left(\frac{X}{Z} ,
\mathcal{N}\right)&:=&\sup_{\{z\in\mathbb{R}\}}\left|\mathbb{P}\left\{\frac{X}{Z} \leq
z\right\}-\mathbb{P}\{\mathcal{N}\leq
z\}\right|\nonumber\\ &\leq& d_{Kol}\left(X ,
\mathcal{N}\right)+\mathbb{P}\{|Z-1|>\varepsilon\}+\varepsilon,\label{ineq-MP}\end{eqnarray}
where $\mathcal{N}$ denotes the standard normal distribution.\\
However, the upper bound  in \eqref{ineq-MP} is usually not sharp
for many  estimators of the drift parameter in stochastic
differential  equations. For instance, we prove that the estimate
\eqref{ineq-MP} does not provide an optimal rate of convergence for
the statistics studied in Section \ref{sect-appl}.

 Recently,
the following observation  provided an improved upper bound in
Kolmogorov distance for normal approximation of   $F_t/G_t$, where
$F_t$ and $G_t$ are random variables.
 Fix $T>0$. Let $f_t,g_t\in \mathcal{H}^{\odot 2}$ for all
$t>0$, and let  $b_t$ be a positive function of $t$ such that
$I_2(g_t)+b_t>0$ almost surely for all $t>0$.
  If $\max_{i=1,2,3}\psi_i(t)\rightarrow 0$ as $t\rightarrow
\infty$, where for every $t>0$,
\begin{eqnarray*}
\psi_1(t)&:=&\frac{1}{b_t^2}\sqrt{\left(b^2_t-2\Vert f_t\Vert^2_{\mathcal{H}^{\otimes 2}}\right)^2+8\Vert f_t\otimes_1 f_t\Vert^2_{\mathcal{H}^{\otimes 2}}},\\
\psi_2(t)&:=&\frac{2}{b_t^2}\sqrt{2 \Vert f_t \otimes_1
g_t\Vert_{\mathcal{H}^{\otimes2}}+\langle
f_t,g_t\rangle^2_{\mathcal{H}^{\otimes 2}}},\\
\psi_3(t)&:=&\frac{2}{b_t^2}\sqrt{\Vert
g_t\Vert^4_{\mathcal{H}^{\otimes 2}} +2\Vert g_t\otimes_1
g_t\Vert^2_{\mathcal{H}^{\otimes 2}}},
\end{eqnarray*}
then, according to    \cite[Corollary 1 ]{kp-JVA}, there exists a positive
constant $C$ such that for all $t$ sufficiently large,
\begin{equation}
d_{Kol}\left(\frac{I_2(f_t)}{ I_2(g_t)+b_t},\mathcal{N}\right) \leq
C \max_{i=1,2,3} \psi_i(t). \label{kp}
\end{equation}
Using  the estimate \eqref{kp}, an optimal Kolmogorov bound in the
CLT for statistics on the basis of continuous observations of the
drift parameter in various SDEs is obtained, see for example, the
papers \cite{EM} and \cite{kp-JVA}. On the other hand, the estimate
\eqref{kp} requires us to evaluate certain terms involving  norms
and contractions which are  not easily computable, and this may
present difficulties. In this context, as far as we know, there is
no result that uses the estimate \eqref{kp} to study the upper bound
 in Kolmogorov distance for normal approximation of
the  estimators based on discrete observations.
\\
The aim of the present work is to provide  a new explicit
Berry-Esseen bound in the Kolmogorov distance for the normal
approximation of  a ratio
 of random variables, which we can also apply to study  parameter estimation
 for discretely observed SDEs, see Theorem \ref{main-thm}. This leads to improved Berry-Esseen bounds of the Kolmogorov distance
for  discretized versions of minimum contrast and maximum likelihood
estimators of the drift parameter   of the Ornstein-Uhlenbeck
process, see Section \ref{sect-appl}. Finally, we mention that in a particular situation, the paper \cite{BBE} provided 
explicit upper bounds for the
Kolmogorov distance for the rate of convergence for the central limit theorem of
 the least squares estimator  
 of the drift parameter of continuously-observed Ornstein–Uhlenbeck processes driven by a Gaussian process with stationary increments.

Our paper is structured as follows: In Section \ref{prelim} we introduce the
necessary notions and notation and recall some important background
material in order to make the paper self-contained. Our main result
is the content of Section \ref{main-sect}. Finally, several applications are developed in
Section    \ref{sect-appl}. In particular, Section \ref{AMCE-sect} deals with an approximate minimum contrast estimator and  Section \ref{AMLE-sect} with approximate maximum likelihood estimators.

\section{Preliminaries}\label{prelim}
This section gives a brief overview of some useful facts from the
Malliavin calculus on Wiener space. Some of the results presented
here are essential for the proofs in the present paper. For our
purposes we focus on special cases that are relevant for our setting
and omit the general high-level theory. We direct the interested
reader to~\cite[Chapter 1]{nualart-book}and~\cite[Chapter
2]{NP-book}.

The first step is to identify the general centered Gaussian process
$(Z_t)_{\geq 0}$ with an   isonormal Gaussian process  $X = \{
X(h), h \in \mathcal{H}\}$ for some Hilbert space $\mathcal{H}$,
that is, $X$ is a centered Gaussian family defined a common
probability space  $(\Omega,
\mathcal{F}, \mathbb{P})$ satisfying, for every
$h_1, h_2 \in \mathcal{H}$,
$\mathbb{E} [ X(h_1) X(h_2) ] = \inner{h_1}{h_2}_\HC$.\\
One can define $\HC$ as the closure of real-valued step functions on
$[0, \infty)$ with respect to the inner product
$\inner{\stepid{t}}{\stepid{s}}_\HC = \mathbb{E}[ Z_t Z_s]$.   Note that
$X(\stepid{t})
\overset{law}{=} Z_t$.\\
The next step involves the  multiple Wiener-It\^o integrals.
The formal definition involves the concepts of Malliavin derivative
and divergence. We refer the reader to~\cite[Chapter
1]{nualart-book}and~\cite[Chapter 2]{NP-book}. For our purposes we
define the multiple Wiener-It\^o integral $I_p$ via the Hermite
polynomials $H_p$. In particular, for $h \in \HC$ with $\norm{h}_\HC
= 1$, and any $p \geq 1$,
\begin{align}
\notag H_p(X(h)) = I_p(h^{\otimes p}).
\end{align}
For $p = 1$ and $p = 2$ we have the following:
\begin{align}
\label{eq:z_rep} H_1(X(\stepid{t})) = & X(\stepid{t}) = I_1(\stepid{t}) \overset{d}{=} Z_t \\
\label{eq:var_rep} H_2(X (\stepid{t})) = & X(\stepid{t})^2 -
E[X(\stepid{t})^2] = I_2(\stepid{t}^{\otimes 2}) \overset{d}{=}
Z_t^2 - \mathbb{E}[Z_t]^2.
\end{align}
Note also that $I_0$ can be taken to be the identity operator.

\noindent $\bullet $ \textbf{Some notation for Hilbert spaces.} Let
$\HC$ be a Hilbert space. Given an integer $q \geq 2$ the Hilbert
spaces $\HC^{\otimes q}$ and $\HC^{\odot q}$ correspond to the $q$th
\emph{tensor product} and $q$th \emph{symmetric tensor product} of
$\HC$. If $f \in \HC^{\otimes q}$ is given by $f = \sum_{j_1,
\ldots, j_q} a(j_1, \ldots, j_q) e_{j_1} \otimes \cdots e_{j_q}$,
where $(e_{j_i})_{i \in [1, q]}$ form an orthonormal basis of
$\HC^{\otimes q}$, then the symmetrization $\tilde{f}$ is given by
\begin{align}
\notag \tilde{f} = \frac{1}{q!} \sum_{\sigma} \sum_{j_1, \ldots,
j_q} a(j_1, \ldots, j_q) e_{\sigma(j_1)} \otimes \cdots
e_{\sigma(j_q)},
\end{align}
where the first sum runs over all permutations  $\sigma$ of $\{1,
\ldots, q\}$. Then $\tilde{f}$ is an element of $\HC^{\odot q}$. We
also make use of the concept of contraction. The $r$th
\emph{contraction} of two tensor products $e_{j_1} \otimes \cdots
\otimes e_{j_p}$ and $e_{k_1} \otimes \cdots e_{k_q}$ is an element
of $\HC^{\otimes (p + q - 2r)}$ given by
\begin{align}
\notag (e_{j_1} & \otimes \cdots \otimes e_{j_p}) \otimes_r (e_{k_1} \otimes \cdots \otimes e_{k_q}) \\
\label{eq:contraction} =  & \quad \left[ \prod_{\ell =1}^r
\inner{e_{j_\ell}}{e_{k_\ell}} \right] e_{j_{r+1}} \otimes \cdots
\otimes e_{j_q} \otimes e_{k_{r+1}} \otimes \cdots \otimes e_{k_q}.
\end{align}

\noindent $\bullet $ \textbf{Isometry property of
integrals~\cite[Proposition 2.7.5]{NP-book}} Fix integers $p, q \geq
1$ as well as $f \in \HC^{\odot p}$ and $g \in \HC^{\odot q}$.
\begin{align}
\label{eq:isometry}  \mathbb{E} [ I_q(f) I_q(g) ] = \left\{
\begin{array}{ll} p! \inner{f}{g}_{\HC^{\otimes p}} & \mbox{ if } p
= q \\ 0 & \mbox{otherwise.} \end{array} \right.
\end{align}

\noindent $\bullet $ \textbf{Product formula~\cite[Proposition
2.7.10]{NP-book}} Let $p,q \geq 1$. If $f \in \HC^{\odot p}$ and $g
\in \HC^{\odot q}$ then
\begin{align}
\label{eq:product} I_p(f) I_q(g) = \sum_{r = 0}^{p \wedge q} r! {p
\choose r} {q \choose r} I_{p + q -2r}(f \widetilde{\otimes}_r g).
\end{align}

\noindent $\bullet $ \textbf{Hypercontractivity in Wiener Chaos.}
For every $q\geq 1$, ${\mathcal{H}}_{q}$ denotes the $q$th Wiener
chaos of $W$, defined as the closed linear subspace of $L^{2}(\Omega
)$ generated by the random variables $\{H_{q}(W(h)),h\in
{{\mathcal{H}}},\Vert h\Vert _{{\mathcal{H}}}=1\}$ where $H_{q}$ is
the $q$th Hermite polynomial. For any $F \in
\oplus_{l=1}^{q}{\mathcal{H}}_{l}$ (i.e. in a fixed sum of Wiener
chaoses), we have
\begin{equation}
\left( \mathbb{E}\big[|F|^{p}\big]\right) ^{1/p}\leqslant c_{p,q}\left(
\mathbb{E}\big[|F|^{2}\big]\right) ^{1/2}\ \mbox{ for any }p\geq 2.
\label{hypercontractivity}
\end{equation}
It should be noted that the constants $c_{p,q}$ above are known with
some precision when $F$ is a single chaos term: indeed,
by~\cite[Corollary 2.8.14]{NP-book}, $c_{p,q}=\left( p-1\right)
^{q/2}$.

\noindent $\bullet $ \textbf{Optimal fourth moment theorem}. Let
$\mathcal{N}$ denote the standard normal law.
Let a sequence $X_{n}\in {\mathcal{H}}_{q}$, such that $\mathbb{E}
X_{n}=0$ and $Var\left[ X_{n}\right] =1$ , and assume $X_{n}$
converges to a normal law in distribution, which is equivalent to
$\lim_{n}\mathbb{E}\left[ X_{n}^{4}\right] =3$. Then we have the  optimal
estimate for total variation distance $d_{TV}\left(
X_n,\mathcal{N}\right) $, known as the optimal 4th moment theorem,
proved in \cite{NP2015} as follows:
 There exist two constants
$c,C>0$ depending only on the sequence $X$ but not on $n$, such that
\begin{equation*}
c\max \left\{ \mathbb{E}\left[ X_{n}^{4}\right] -3,\left\vert \mathbb{E}
\left[ X_{n}^{3}\right] \right\vert \right\} \leqslant d_{TV}\left(
X_n,\mathcal{N}\right) \leqslant C\max \left\{ \mathbb{E}\left[
X_{n}^{4}\right] -3,\left\vert \mathbb{E}\left[ X_{n}^{3}\right] \right\vert
\right\}.
\end{equation*}
On the other hand, since $d_{Kol}\left(.,.\right)\leq d_{TV}\left(.,.\right)$, we have 
\begin{equation}
  d_{Kol}\left(
X_n,\mathcal{N}\right) \leqslant C\max \left\{ \mathbb{E}\left[
X_{n}^{4}\right] -3,\left\vert \mathbb{E}\left[ X_{n}^{3}\right] \right\vert
\right\}.\label{fourth-cumulant-thm}
\end{equation}
 Moreover, we recall that   the third and fourth
cumulants are respectively
$$
\begin{aligned}
&\kappa_{3}(X)=\mathbb{E}\left[X^{3}\right]-3 \mathbb{E}\left[X^{2}\right] \mathbb{E}[X]+2 \mathbb{E}[X]^{3}, \\
&\kappa_{4}(X)=\mathbb{E}\left[X^{4}\right]-4 \mathbb{E}[X] \mathbb{E}\left[X^{3}\right]-3
\mathbb{E}\left[X^{2}\right]^{2}+12 \mathbb{E}[X]^{2} \mathbb{E}\left[X^{2}\right]-6 \mathbb{E}[X]^{4}.
\end{aligned}
$$
In particular, when $\mathbb{E}[X]=0$, we have that
\[\kappa_{3}(X)=\mathbb{E}\left[X^{3}\right]\ \mbox{ and }\ \kappa_{4}(X)=
\mathbb{E}\left[X^{4}\right]-3 \mathbb{E}\left[X^{2}\right]^{2}.\]

Throughout the paper we use the notation $\mathcal{N} \sim
\mathcal{N}(0,1)$. We also use the notation $C$ for any positive
real constant, independently of its value which may change from line
to line when this does not lead to ambiguity.\\

\section{Main result}\label{main-sect}

In this section we present in detail our main result, which is
stated in full generality in the forthcoming Theorem \ref{main-thm}.

 Let us start with the following assumptions.

\noindent \textbf{Assumption} $\mathbf{\left(\mathcal{A}_1\right)}$:
Let $q$ be a fixed positive integer, and let $\{G_T,T>0\}$ be
a  stochastic process satisfying
\begin{eqnarray}
\left|\frac{1}{\rho \sqrt{T}}\mathbb{E}G_T-1\right|\rightarrow0,\mbox{ as } T\rightarrow\infty,\qquad
\mathbb{E}\left[(G_T-\mathbb{E}G_T)^2\right]\longrightarrow\sigma^2\
\mbox{ as } T\rightarrow\infty\label{estimate-EG_T}
\end{eqnarray}
for some positive constants $\rho>0,\ \sigma>0$, and
\begin{eqnarray}&G_T-\mathbb{E}G_T=V_T+\frac{1}{\sqrt{T}}R_T,\quad T>0,\label{relation-D-V-R}
\end{eqnarray}
where  \begin{eqnarray}V_T,R_T\in{\mathcal{H}}_{q},\ T>0, \mbox{ such that  } \frac{\|R_T\|_{L^2(\Omega)}}{\sqrt{T}}\rightarrow0 \mbox{ as } T\rightarrow\infty.\label{condition-V-R}\end{eqnarray}
\noindent \textbf{Assumption} $\mathbf{\left(\mathcal{A}_2\right)}$: Let   $\{(A_T,a_T),T>0\}$ be
a  stochastic process satisfying
\begin{eqnarray}  A_T \in {\mathcal{H}}_{q},
\mbox{ and } a_T\in \mathbb{R} \mbox{ such that  } \frac{\|A_T\|_{L^2(\Omega)}+|a_T|}{\sqrt{T}}\rightarrow0 \mbox{ as } T\rightarrow\infty.\label{condition-A-a}\end{eqnarray}
 Our purpose is to
provide an upper bound
 in Kolmogorov distance for normal approximation of
the ratio
\begin{eqnarray}\frac{\frac{1}{\sigma}(G_T-\mathbb{E}G_T)+\frac{1}{\sqrt{T}}\left(A_T+a_T\right)}{\frac{1}{\rho
\sqrt{T}}G_{T}},\quad T>0,\label{def-Q}\end{eqnarray} where $\{G_T,T>0\}$  and $\{(A_T,a_T),T>0\}$
  are stochastic processes  satisfying $\left(\mathcal{A}_1\right)$ and $\mathbf{\left(\mathcal{A}_2\right)}$, respectively. \\
 For every $z \in \mathbb{R}$,
$T>0$, let
$$
H_T(z):=\frac{ \left(\frac{1}{\sigma}-\frac{z}{\rho
\sqrt{T}}\right)(G_T-\mathbb{E}G_T)+\frac{1}{\sqrt{T}}\left(A_T+a_T\right)}{\frac{1}{\rho
\sqrt{T}}\mathbb{E}G_T}.
$$
Setting
\begin{eqnarray}M_T(z):=
\left(\frac{1}{\sigma}-\frac{z}{\rho
\sqrt{T}}\right)(G_T-\mathbb{E}G_T)+\frac{1}{\sqrt{T}}A_T,\quad z
\in \mathbb{R},\ T>0, \label{exp-M}\end{eqnarray} we have
\begin{eqnarray*}H_T(z)=\frac{M_T(z)+\frac{1}{\sqrt{T}} a_T}{\frac{1}{\rho \sqrt{T}}\mathbb{E}G_T},\quad z \in \mathbb{R},\
T>0.\label{exp-H}\end{eqnarray*} Then, for every $z \in \mathbb{R}$,
$T>0$,
\begin{eqnarray}\left\{\frac{\frac{1}{\sigma}(G_T-\mathbb{E}G_T)+\frac{1}{\sqrt{T}}\left(A_T+a_T\right)}{\frac{1}{\rho
\sqrt{T}}G_{T}} \leq z\right\}=\{H_T(z) \leq
z\}.\label{link-Q-H}\end{eqnarray}

\begin{theorem}\label{main-thm} Let $\{G_T,T>0\}$  and $\{(A_T,a_T),T>0\}$
  are stochastic processes  satisfying $\left(\mathcal{A}_1\right)$ and $\mathbf{\left(\mathcal{A}_2\right)}$, respectively. Then, there exist
a constant $C>0$ (independent of $T$) such that, for all $T>0$,
\begin{eqnarray}
&&d_{Kol}\left(\frac{\frac{1}{\sigma}(G_T-\mathbb{E}G_T)+\frac{1}{\sqrt{T}}\left(A_T+a_T\right)}{\frac{1}{\rho
\sqrt{T}}G_{T}},  \mathcal{N}  \right)\nonumber\\&\leq&
\max\left(\left|\kappa_{3}\left(\frac{V_T}{\sigma}\right)\right|, \kappa_{4}\left(\frac{V_T}{\sigma}\right) \right)
+C{T^{\frac14}}\left|\frac{1}{\rho
\sqrt{T}}\mathbb{E}G_T-1\right|+ C\left| \mathbb{E}[( G_T-\mathbb{E} G_T)^2] -
\sigma^2\right|\nonumber\\&&+\frac{C}{\sqrt{T}}\left(\|R_T\|_{L^2(\Omega)}+\|A_T\|_{L^2(\Omega)}+|a_T|\right),\label{kol-bound-F}
\end{eqnarray}
where the constants  $\rho$ and $\sigma$ are defined by
\eqref{estimate-EG_T}, and $\{V_T,T>0\}$ and $\{R_T,T>0\}$ are  the processes  given
by \eqref{relation-D-V-R}.
\end{theorem}
\begin{proof}
Denote
\begin{eqnarray*}Q_T:=\frac{\frac{1}{\sigma}(G_T-\mathbb{E}G_T)+\frac{1}{\sqrt{T}}\left(A_T+a_T\right)}{\frac{1}{\rho
\sqrt{T}}G_{T}},\quad T>0.\end{eqnarray*}
 If $z<0$,
\begin{eqnarray}
\left|\mathbb{P}\left\{Q_T \leq z\right\}-\mathbb{P}\{\mathcal{N}
\leq z\}\right|&=& \left|\mathbb{P}\left\{- Q_T \geq
-z\right\}-\mathbb{P}\{\mathcal{N} \geq -z\}\right|\nonumber\\&=&
\left|\mathbb{P}\left\{-Q_T \leq -z\right\}-\mathbb{P}\{\mathcal{N}
\leq -z\}\right|.\label{kol-case-z<0}
\end{eqnarray}
Now, let $T>\frac{1}{\rho ^4}$ and
$0<z\leq\frac{T^{\frac14}}{\sigma}$. In this case,  $\frac{1}{\sigma}-\frac{z}{\rho \sqrt{T}}>0$. By \eqref{estimate-EG_T},
\eqref{condition-A-a} and \eqref{exp-M}, we get
\begin{eqnarray}\mathbb{E}\left[M_T(z)^2\right]&=& \left(\frac{1}{\sigma}-\frac{z}{\rho \sqrt{T}}\right)^2
\mathbb{E}\left[(G_T-\mathbb{E}G_T)^2\right]+\frac{1}{T}\mathbb{E}\left[\left(A_T+a_T\right)^2\right]
\nonumber\\&&+2\left(\frac{1}{\sigma}-\frac{z}{\rho \sqrt{T}}\right)
\frac{\mathbb{E}\left[\left(A_T+a_T\right)(G_T-\mathbb{E}G_T)\right]}{\sqrt{T}}\label{decomp-K}\\
&\geq& \frac{1}{\sigma^2}\left(1-\frac{1}{\rho T^{\frac14}}\right)^2
\mathbb{E}\left[(G_T-\mathbb{E}G_T)^2\right]+\frac{1}{T}\mathbb{E}\left[\left(A_T+a_T\right)^2\right]
\nonumber\\&&-2 \frac{1}{\sigma}\left|1-\frac{1}{\rho
T^{\frac14}}\right|
\frac{\left|\mathbb{E}\left[\left(A_T+a_T\right)(G_T-\mathbb{E}G_T)\right]\right|}{\sqrt{T}}\nonumber\\
&=:&\eta_T\longrightarrow 1\quad \mbox{ as }
T\rightarrow\infty.\nonumber
\end{eqnarray}
This  implies
\begin{eqnarray}c_2>\sup_{T>T_0,\ 0<z\leq\frac{T^{\frac14}}{\sigma}}\mathbb{E}\left[M_T(z)^2\right]\geq
 \sup_{T>T_0} \eta_T>c_1>0 \label{sup of M}\end{eqnarray}
for some $T_0>\frac{1}{\rho ^4}$.\\
On the other hand, using \eqref{link-Q-H},
\begin{eqnarray}
 &&\left|\mathbb{P}\left\{Q_T \leq z\right\}-\mathbb{P}\{\mathcal{N}
\leq z\}\right|=\left|\mathbb{P}\left\{H_T(z) \leq
z\right\}-\mathbb{P}\{\mathcal{N} \leq
z\}\right|\nonumber\\&\leq&\left|\mathbb{P}\left\{\frac{M_T\left(z\right)}{\sqrt{\mathbb{E}M_T\left(z\right)^2}}
\leq \frac{z\frac{1}{\rho
\sqrt{T}}\mathbb{E}G_T-\frac{1}{\sqrt{T}}a_T}{\sqrt{\mathbb{E}M_T\left(z\right)^2}}\right\}
-\mathbb{P}\left\{\mathcal{N} \leq \frac{z\frac{1}{\rho
\sqrt{T}}\mathbb{E}G_T-\frac{1}{\sqrt{T}}a_T}{\sqrt{\mathbb{E}M_T\left(z\right)^2}}\right\}\right|\nonumber\\&&+
\left|\mathbb{P}\left\{\mathcal{N}\leq \frac{z\frac{1}{\rho
\sqrt{T}}\mathbb{E}G_T-\frac{1}{\sqrt{T}}a_T}{\sqrt{\mathbb{E}M_T\left(z\right)^2}}\right\}
-\mathbb{P}\left\{\mathcal{N} \leq \frac{z\frac{1}{\rho
\sqrt{T}}\mathbb{E}G_T}{\sqrt{\mathbb{E}M_T\left(z\right)^2}}\right\}\right|\nonumber\\&&+
\left|\mathbb{P}\left\{\mathcal{N}\leq \frac{z\frac{1}{\rho
\sqrt{T}}\mathbb{E}G_T}{\sqrt{\mathbb{E}M_T\left(z\right)^2}}\right\}
-\mathbb{P}\{\mathcal{N} \leq z\}\right|.\label{triang.-kol}
\end{eqnarray}
Furthermore, using  \eqref{estimate-EG_T}, \eqref{condition-A-a},
\eqref{sup of M}, \eqref{triang.-kol}  and the fact that
\begin{eqnarray}\left|\mathbb{P}\left\{\mathcal{N} \leq
x\right\}-\mathbb{P}\left\{\mathcal{N} \leq
y\right\}\right|\leq\frac{|x-y|}{\sqrt{2\pi}}e^{-\frac{\min(x^2,y^2)}{2}},\quad
x,y \mbox{ in } \mathbb{R},\label{diff-proba-estimate}\end{eqnarray}
and
\begin{eqnarray}\sup_{\{z>0\}}|z|e^{-C z^2}<\infty,\qquad
 \sup_{\{z>0\}}z^2e^{-C z^2}<\infty,\label{sup-z-exp}
\end{eqnarray}
 we obtain, for
every $0<z\leq\frac{T^{\frac14}}{\sigma},\ T>T_0$,
\begin{eqnarray}
 &&\left|\mathbb{P}\left\{Q_T \leq
 z\right\}-\mathbb{P}\{\mathcal{N} \leq z\}\right| \nonumber\\&\leq&
d_{{Kol}}\left(\frac{M_T\left(z\right)}{\sqrt{\mathbb{E}M_T\left(z\right)^2}}
,\mathcal{N}\right)+\frac{1}{\sqrt{2\pi}}\left|\frac{\frac{1}{\sqrt{T}}a_T}{\sqrt{\mathbb{E}M_T\left(z\right)^2}}\right|\nonumber\\
&&+\frac{|z|}{\sqrt{2\pi}}\left|\frac{\frac{1}{\rho
\sqrt{T}}\mathbb{E}G_T}{\sqrt{\mathbb{E}M_T\left(z\right)^2}}-1\right|
e^{-C z^2 \min\left(1,\frac{\frac{1}{\rho
\sqrt{T}}\mathbb{E}G_T}{\sqrt{\mathbb{E}M_T\left(z\right)^2}}\right)^2}\nonumber\\
&\leq&
d_{{Kol}}\left(\frac{M_T\left(z\right)}{\sqrt{\mathbb{E}M_T\left(z\right)^2}}
,\mathcal{N}\right)+C{T^{\frac14}}\left|\frac{1}{\rho
\sqrt{T}}\mathbb{E}G_T-1\right|+ C\left| \mathbb{E}[( G_T-\mathbb{E} G_T)^2] -
\sigma^2\right|\nonumber\\&&+\frac{C}{\sqrt{T}}\left(\|A_T\|_{L^2(\Omega)}+|a_T|\right),\label{{Kol}-F-z<}
\end{eqnarray}
where we   used, according to \eqref{condition-A-a} and \eqref{sup
of M},
\[
  \frac{1}{\sqrt{2\pi}}\left|\frac{\frac{1}{\sqrt{T}}a_T}{\sqrt{\mathbb{E}M_T\left(z\right)^2}}\right|
\leq \frac{C|a_T|}{\sqrt{T}},
\]
and, by \eqref{estimate-EG_T}, \eqref{decomp-K}, \eqref{sup of M}
and \eqref{sup-z-exp},
\begin{eqnarray*}&&\frac{|z|}{\sqrt{2\pi}}\left|\frac{\frac{1}{\rho
\sqrt{T}}\mathbb{E}G_T}{\sqrt{\mathbb{E}M_T\left(z\right)^2}}-1\right|
e^{-C z^2 \min\left(1,\frac{\frac{1}{\rho
\sqrt{T}}\mathbb{E}G_T}{\sqrt{\mathbb{E}M_T\left(z\right)^2}}\right)^2}\\&\leq&
\frac{|z|}{\sqrt{2\pi}}\left|\frac{\frac{1}{\rho
\sqrt{T}}\mathbb{E}G_T-1}{\sqrt{\mathbb{E}M_T\left(z\right)^2}}\right|
e^{-C z^2 \min\left(1,\frac{\frac{1}{\rho
\sqrt{T}}\mathbb{E}G_T}{\sqrt{\mathbb{E}M_T\left(z\right)^2}}\right)^2}
\\&&+
\frac{|z|}{\sqrt{2\pi}}\left|\frac{1}{\sqrt{\mathbb{E}M_T\left(z\right)^2}}-1\right|
e^{-C \left(\frac{z}{\sqrt{\mathbb{E}M_T\left(z\right)^2}}\right)^2
\min\left(\sqrt{\eta_T}, \frac{1}{\rho \sqrt{T}}\mathbb{E}G_T
\right)^2}
\\&\leq& C  |z|\left|\frac{1}{\rho
\sqrt{T}}\mathbb{E}G_T-1\right| + C\frac{|z|}{\sqrt{\mathbb{E}M_T\left(z\right)^2}}\left| \mathbb{E}M_T\left(z\right)^2 -1\right|
e^{-C \left(\frac{z}{\sqrt{\mathbb{E}M_T\left(z\right)^2}}\right)^2}
\\&\leq&C{T^{\frac14}}\left|\frac{1}{\rho
\sqrt{T}}\mathbb{E}G_T-1\right| +  C\left| \mathbb{E}[( G_T-\mathbb{E} G_T)^2] -
\sigma^2\right|+\frac{C}{\sqrt{T}}\left(\|A_T\|_{L^2(\Omega)}+|a_T|\right).\end{eqnarray*}
 On the other
hand, using  the binomial expansion, \eqref{estimate-EG_T}--\eqref{condition-A-a} and \eqref{exp-M}, we
get, for every $0<z\leq\frac{T^{\frac14}}{\sigma},\ T>T_0$,
\begin{eqnarray}
\kappa_{4}(M_T(z))&=&\left|\kappa_{4}(M_T(z))\right|=\left|\mathbb{E}(M_T(z)^4)-3(\mathbb{E}(M_T(z)^2))^2\right|\nonumber\\
&\leq&\left|\frac{1}{\sigma}-\frac{z}{\rho
\sqrt{T}}\right|^4\left|\mathbb{E}((G_T-\mathbb{E}G_T)^4)-3(\mathbb{E}((G_T-\mathbb{E}G_T)^2))^2\right|
+\frac{C}{\sqrt{T}}\left(\|A_T\|_{L^2(\Omega)}+|a_T|\right)\nonumber\\
&\leq&\left|\kappa_{4}(\frac{1}{\sigma}(G_T-\mathbb{E}G_T))\right|+\frac{C}{\sqrt{T}}\left(\|A_T\|_{L^2(\Omega)}+|a_T|\right)\nonumber\\
&\leq&\kappa_{4}\left(\frac{V_T}{\sigma}\right)+\frac{C}{\sqrt{T}}\left(\|R_T\|_{L^2(\Omega)}+\|A_T\|_{L^2(\Omega)}+|a_T|\right).\label{4cumulant-F-A}
\end{eqnarray}
Similarly,  for every $0<z\leq\frac{T^{\frac14}}{\sigma},\ T>T_0$,
\begin{eqnarray}
\left|\kappa_{3}(M_T(z))\right|&=&\left|\mathbb{E}(M_T(z)^3)\right|\nonumber\\
&\leq&\left|\frac{1}{\sigma}-\frac{z}{\rho
\sqrt{T}}\right|^3\left|\mathbb{E}((G_T-\mathbb{E}G_T)^3)\right|
+\frac{C}{\sqrt{T}}\left(\|A_T\|_{L^2(\Omega)}+|a_T|\right)\nonumber\\
&\leq&\left|\kappa_{3}(\frac{1}{\sigma}(G_T-\mathbb{E}G_T))\right|+\frac{C}{\sqrt{T}}\left(\|A_T\|_{L^2(\Omega)}+|a_T|\right)\nonumber\\
&\leq&\left|\kappa_{3}\left(\frac{V_T}{\sigma}\right)\right|+\frac{C}{\sqrt{T}}\left(\|R_T\|_{L^2(\Omega)}+\|A_T\|_{L^2(\Omega)}+|a_T|\right).\label{3cumulant-F-A}
\end{eqnarray}
Combining \eqref{fourth-cumulant-thm}, \eqref{{Kol}-F-z<},
\eqref{4cumulant-F-A} and \eqref{3cumulant-F-A}, we deduce
\begin{eqnarray}
&&\sup_{\{0\leq
z\leq\frac{T^{\frac14}}{\sigma}\}}\left|\mathbb{P}\left\{Q_T \leq
z\right\}-\mathbb{P}\{\mathcal{N} \leq z\}\right|\nonumber\\&&\leq
\max\left(\left|\kappa_{3}\left(\frac{V_T}{\sigma}\right)\right|,\kappa_{4}\left(\frac{V_T}{\sigma}\right)\right)
+C{T^{\frac14}}\left|\frac{1}{\rho
\sqrt{T}}\mathbb{E}G_T-1\right|\nonumber\\&&
+ C\left| \mathbb{E}[( G_T-\mathbb{E} G_T)^2] -
\sigma^2\right| +\frac{C}{\sqrt{T}}\left(\|R_T\|_{L^2(\Omega)}+\|A_T\|_{L^2(\Omega)}+|a_T|\right).\label{{Kol}-caseII-z>}
\end{eqnarray}
Let $z> \frac{T^{\frac14}}{\sigma}$. By \eqref{link-Q-H}, and the
fact that  $\mathbb{P}\left\{\mathcal{N} \geq z\right\}\leq C
e^{-\frac{z^2}{4}}$ for all $z>0$, we have
\begin{eqnarray}
\left|\mathbb{P}\left\{Q_T \leq z\right\}-\mathbb{P}\{\mathcal{N}
\leq z\}\right| &\leq& \mathbb{P}\left\{Q_T
  \geq \frac{T^{\frac14}}{\sigma}\right\}+\mathbb{P}\left\{\mathcal{N} \geq \frac{T^{\frac14}}{\sigma}\right\}\nonumber \\
&\leq& \left|\mathbb{P}\left\{Q_T \leq
\frac{T^{\frac14}}{\sigma}\right\}-\mathbb{P}\left\{\mathcal{N} \leq
\frac{T^{\frac14}}{\sigma}\right\}\right|
+2 \mathbb{P}\left\{\mathcal{N} \geq \frac{T^{\frac14}}{\sigma}\right\}\nonumber \\
 &\leq& \left|\mathbb{P}\left\{H_T\left(\frac{T^{\frac14}}{\sigma}\right) \leq
\frac{T^{\frac14}}{\sigma}\right\}-\mathbb{P}\left\{\mathcal{N} \leq
\frac{T^{\frac14}}{\sigma}\right\}\right|+\frac{C}{e^{\frac{\sqrt{T}}{4\sigma^2}}}.\label{z-greater}
\end{eqnarray}
Moreover,
\begin{eqnarray*}
&& \left|\mathbb{P}\left\{H_T\left(\frac{T^{\frac14}}{\sigma}\right)
\leq \frac{T^{\frac14}}{\sigma}\right\}-\mathbb{P}\left\{\mathcal{N}
\leq \frac{T^{\frac14}}{\sigma}\right\}\right|\nonumber\\
&\leq&\left|\mathbb{P}\left\{\frac{M_T\left(\frac{T^{\frac14}}{\sigma}\right)}{\sqrt{\mathbb{E}\left[M_T\left(\frac{T^{\frac14}}{\sigma}\right)^2\right]}}
\leq \frac{\frac{1}{\sigma\rho
T^{\frac14}}\mathbb{E}G_T-\frac{1}{\sqrt{T}}a_T}{\sqrt{\mathbb{E}\left[M_T\left(\frac{T^{\frac14}}{\sigma}\right)^2\right]}}\right\}
-\mathbb{P}\left\{\mathcal{N} \leq \frac{\frac{1}{\sigma\rho
T^{\frac14}}G_T-\frac{1}{\sqrt{T}}a_T}{\sqrt{\mathbb{E}\left[M_T\left(\frac{T^{\frac14}}{\sigma}\right)^2\right]}}\right\}\right|\nonumber\\&&+
\left|\mathbb{P}\left\{\mathcal{N}\leq \frac{\frac{1}{\sigma\rho
T^{\frac14}}\mathbb{E}G_T-\frac{1}{\sqrt{T}}a_T}{\sqrt{\mathbb{E}\left[M_T\left(\frac{T^{\frac14}}{\sigma}\right)^2\right]}}\right\}
-\mathbb{P}\left\{\mathcal{N} \leq \frac{\frac{1}{\sigma\rho
T^{\frac14}}\mathbb{E}G_T}{\sqrt{\mathbb{E}\left[M_T\left(\frac{T^{\frac14}}{\sigma}\right)^2\right]}}\right\}\right|\nonumber\\&&+
\left|\mathbb{P}\left\{\mathcal{N}\leq \frac{\frac{1}{\sigma\rho
T^{\frac14}}\mathbb{E}G_T}{\sqrt{\mathbb{E}\left[M_T\left(\frac{T^{\frac14}}{\sigma}\right)^2\right]}}\right\}
-\mathbb{P}\left\{\mathcal{N} \leq
\frac{T^{\frac14}}{\sigma}\right\}\right|.
\end{eqnarray*}
Combining this fact with \eqref{estimate-EG_T},
\eqref{condition-A-a}, \eqref{sup of M} and
\eqref{diff-proba-estimate},
\begin{eqnarray}
&& \left|\mathbb{P}\left\{H_T\left(\frac{T^{\frac14}}{\sigma}\right)
\leq \frac{T^{\frac14}}{\sigma}\right\}-\mathbb{P}\left\{\mathcal{N}
\leq \frac{T^{\frac14}}{\sigma}\right\}\right|\nonumber\\
&\leq&
d_{{Kol}}\left(\frac{M_T\left(\frac{T^{\frac14}}{\sigma}\right)}{\sqrt{\mathbb{E}\left[M_T\left(\frac{T^{\frac14}}{\sigma}\right)^2\right]}}
,\mathcal{N}\right)+
C\left|\frac{\frac{1}{\sqrt{T}}a_T}{\sqrt{\mathbb{E}\left[M_T\left(\frac{T^{\frac14}}{\sigma}\right)^2\right]}}\right|
+  \left|\frac{\frac{\mathbb{E}G_{T}}{\rho
\sqrt{T}}}{\sqrt{\mathbb{E}\left[M_T\left(\frac{T^{\frac14}}{\sigma}\right)^2\right]}}
-1\right|\frac{T^{\frac14}e^{-C\frac{T^{\frac12}}{\sigma^2}}}{\sqrt{2\pi}\sigma}\nonumber  \\
&\leq&
d_{{Kol}}\left(\frac{M_T\left(\frac{T^{\frac14}}{\sigma}\right)}{\sqrt{\mathbb{E}\left[M_T\left(\frac{T^{\frac14}}{\sigma}\right)^2\right]}}
,\mathcal{N}\right)
 +\frac{C|a_T|}{\sqrt{T}}.\label{{Kol}-case-z>}
\end{eqnarray}
Thus, by \eqref{fourth-cumulant-thm}, \eqref{4cumulant-F-A},
\eqref{3cumulant-F-A}, \eqref{z-greater} and \eqref{{Kol}-case-z>}, we
obtain
\begin{eqnarray}
 \sup_{\{z>\frac{T^{\frac14}}{\sigma}\}}\left|\mathbb{P}\left\{Q_T
\leq z\right\}-\mathbb{P}\{\mathcal{N} \leq z\}\right| \leq
\max\left(\left|\kappa_{3}\left(\frac{V_T}{\sigma}\right)\right|,\kappa_{4}\left(\frac{V_T}{\sigma}\right)\right)
+\frac{C\left(\|R_T\|_{L^2(\Omega)}+|a_T|\right)}{\sqrt{T}}.\label{{Kol}-caseI-z>}
\end{eqnarray}
 Consequently, using \eqref{kol-case-z<0}, \eqref{{Kol}-caseII-z>} and \eqref{{Kol}-caseI-z>}, we deduce
\begin{eqnarray*}
&&\sup_{\{z\in\mathbb{R}\}}\left|\mathbb{P}\left\{Q_T \leq
z\right\}-\mathbb{P}\{\mathcal{N} \leq z\}\right|\nonumber\\&&\leq
\max\left(\left|\kappa_{3}\left(\frac{V_T}{\sigma}\right)\right|,\kappa_{4}\left(\frac{V_T}{\sigma}\right)\right)
+C{T^{\frac14}}\left|\frac{1}{\rho
\sqrt{T}}\mathbb{E}G_T-1\right|\nonumber\\&&
+ C\left| \mathbb{E}[( G_T-\mathbb{E} G_T)^2] -
\sigma^2\right| +\frac{C}{\sqrt{T}}\left(\|R_T\|_{L^2(\Omega)}+\|A_T\|_{L^2(\Omega)}+|a_T|\right),
\end{eqnarray*}
which proves \eqref{kol-bound-F}. Therefore the proof is complete.
\end{proof}

\section{Applications}\label{sect-appl}

We now apply our approach to   Ornstein-Uhlenbeck (OU) processes,
and we compare our findings  with the relevant results in the
literature.

Let $X:=\left\{X_t, t \geq 0\right\}$  be the OU process driven by a
Brownian motion $\left\{W_{t},t\geq 0\right\} $.  More precisely,
$X$ is the solution of the following linear stochastic differential
equation
\begin{equation}
X_{0}=0;\quad dX_{t}=-\theta X_{t}dt+dW_{t},\quad t\geq 0,
\label{INTRO-OU}
\end{equation}
where $\theta >0$ is an unknown parameter.\\
The drift parametric estimation for the OU process \eqref{INTRO-OU}
has been widely studied in the literature.
 There are several  methods that can estimate the parameter $\theta$
 in \eqref{INTRO-OU} such as maximum likelihood estimation, least squares
 estimation and  minimum
contrast estimation, we refer to  monographs \cite{kutoyants,LS}.
Here we consider  discretized versions of minimum contrast and
maximum likelihood estimators of the drift parameter $\theta$  of
the OU process \eqref{INTRO-OU}. We suppose that the process $X$
given by \eqref{INTRO-OU}, is observed equidistantly in time with
the step size $\Delta_n$ : $t_i=i \Delta_n, i=0, \cdots, n$, and
$T=n \Delta_n$ denotes the length of the ``observation window".

Let us now introduce the necessary notions and notation and recall
some important background material needed for the proofs in this
section.\\
 Since \eqref{INTRO-OU} is
linear, it is immediate to see that its solution can be expressed
  explicitly as
\begin{align}
X_{t}=\int_{0}^{t}e^{-\theta (t-s)}dW_{s}. \label{OUX}
\end{align}
Moreover, It is well-known  that
\begin{equation}
Z_{t}=\int_{-\infty }^{t}e^{-\theta (t-s)}dW_{s}\label{OUZ}
\end{equation}
 is a stationary Gaussian process, and
\begin{equation}
E[Z_{t}^2]= E[Z_{0}^2]=\frac{1}{2\theta},\quad t>0.\label{var-Z}
\end{equation}
 Furthermore,
 \begin{equation}
X_{t}=Z_{t}-e^{-\theta t}Z_0.\label{X-Z}
\end{equation}  
Let us introduce
\[F_n(Z):=\sqrt{T}\left(f_n(Z)-\frac{1}{2\theta}\right), \mbox{ where } f_n(Z) := \frac{1}{n} \sum_{i =0}^{n-1}
Z_{t_{i}}^{2}.\]
We will make use of the following  technical lemmas.
\begin{lemma} Let $X$ and $Z$ be the processes given in \eqref{OUX}
and \eqref{OUZ} respectively. Then there exists $C>0$ depending only
on $\theta$ such that  for every $p \geqslant 1$ and for all $n
\geq1$,
\begin{equation}
\left\|f_n(X)-f_n(Z)\right\|_{L^p(\Omega)}\leq
\frac{C}{n\Delta_n}.\label{error X-Z}
\end{equation}
\end{lemma}
\begin{proof}Using \eqref{X-Z}, we can write
\begin{eqnarray}
\left\|f_n(X)-f_n(Z)\right\|_{L^p(\Omega)}\leq \frac{1}{n} \sum_{i
=0}^{n-1}\left\|e^{-2\theta t_{i}}Z_{0}^{2}-2e^{-\theta
t_{i}}Z_{0}Z_{t_{i}}\right\|_{L^p(\Omega)}.\label{difference_fX-fZ}
\end{eqnarray}
Since $Z$ is a stationary Gaussian process, then, using
\eqref{difference_fX-fZ} and
$\frac{\Delta_n}{1-e^{-\theta\Delta_n}}\rightarrow\frac{1}{\theta}$
as $n\rightarrow\infty$, we obtain
\begin{eqnarray*}
\left\|f_n(X)-f_n(Z)\right\|_{L^p(\Omega)}&\leq& \frac{C}{n} \sum_{i
=0}^{n-1} e^{-\theta t_{i}}  = \frac{C}{n}
\frac{1-e^{-n\theta\Delta_n}}{1-e^{-\theta\Delta_n}}  \leq
\frac{C}{n\Delta_n},
\end{eqnarray*}
which completes the proof.\end{proof}

\begin{lemma}[\cite{EAA2023}] Let $X$ and $Z$ be the processes given in \eqref{OUX}
and \eqref{OUZ} respectively.
There exists $C>0$ depending only on $\theta$ such that for large
$n$
\begin{align}
\left|E\left(F_n^2(Z)\right)-\frac{1}{2\theta^3}\right|&\leq
C\left(\Delta_n^2+\frac{1}{n\Delta_n}\right).\label{variance-F(Z)}
\end{align}
Consequently, using \eqref{error X-Z}, for large $n$
\begin{align}
\left|E\left(F_n^2(X)\right)-\frac{1}{2\theta^3}\right|&\leq
C\left(\Delta_n^2+\frac{1}{n\Delta_n}\right).\label{variance-F(X)}
\end{align}
\end{lemma}

\begin{lemma}[\cite{EAA2023}]Let   $Z$ be the process  given in  \eqref{OUZ}. There exists $C>0$ depending only on $\theta$ such that  for large
$n$,
\begin{eqnarray}
|k_{3}(F_n(Z))|&\leq&\frac{C}{\left(n\Delta_n\right)^{3/2}},\label{k3-estimator1}
\end{eqnarray}
\begin{eqnarray} \left|k_{4}(F_n(Z))\right|\leq C\frac{1}{n\Delta_n}.\label{k4-estimator1}
\end{eqnarray}
Consequently,
\begin{eqnarray} \max\left(|k_{3}(F_n(Z))|,\left|k_{4}(F_n(Z))\right|\right)\leq C\frac{1}{n\Delta_n}.\label{max(k3,k4)-estimator1}
\end{eqnarray}
\end{lemma}
Let $X$ be the process given by \eqref{INTRO-OU}, and let us define the
following sequences
$$
S_{n}:=\Delta_{n} \sum_{i=1}^{n} X_{t_{i-1}}^{2}=Tf_n(X),
$$
and
$$
\Lambda_{n}:=\sum_{i=1}^{n} e^{-\theta
t_{i}}X_{t_{i-1}}\left(\zeta_{t_{i}}-\zeta_{t_{i-1}}\right)=\sum_{i=1}^{n}
e^{-\theta(t_{i}+t_{i-1})}\zeta_{t_{i-1}}\left(\zeta_{t_{i}}-\zeta_{t_{i-1}}\right),
$$
where
\[\zeta_{t}=\int_{0}^{t}e^{\theta s}dW_{s}. \label{zeta}\]
It is easy to see that 
\begin{eqnarray}
\sum_{i=1}^{n} X_{t_{i-1}}\left(X_{t_{i}}-X_{t_{i-1}}\right)=\frac{e^{-\theta
\Delta_{n}}-1}{\Delta_{n}}S_{n}+ \Lambda_{n}.\label{decomp-Num in Gamma+S}
\end{eqnarray}
According to \cite{EAA2023},
there exists $C>0$ depending only on $\theta$ such that for large
$n$
\begin{align}
\left|E\left[\left(\frac{1}{\sqrt{T}}
\Lambda_{n}\right)^2\right]-\frac{1}{2\theta}\right|&\leq
C\left(\Delta_n+\frac{1}{n\Delta_n}\right).\label{cv-Lambda}
\end{align}

\subsection{Approximate minimum contrast estimator}\label{AMCE-sect}

In this section we    provide upper bounds
in the  Kolomogorov distance  for the rate of normal convergence of
an approximate minimum contrast estimator of the drift parameter
$\theta$ of the OU process $X $ given by \eqref{INTRO-OU}.

Here we are concerned with the approximate minimum contrast
estimator (AMCE)
\[\widetilde{\theta}_{n}:= \frac{1}{\frac{2}{n}
\sum_{i=1}^{n}X_{t_{i}}^{2}},\]
  which is a
 discrete version  of the minimum contrast estimator
(MCE)   defined as  
\[
 \breve{\theta} _{T}:=\frac{1}{\frac{2}{T}\int_{0}^{T} X_{s}^{2}
\mathrm{~d} s},
\quad T\geq 0.
\]
Recall that, for two random variables $X$ and $Y$, the Wasserstein
metric is given by
\begin{eqnarray*}
 d_{W}\left( X,Y\right) := \sup_{f\in
Lip(1)}\left\vert E [f(X)]-E [f(Y)]\right\vert,
\end{eqnarray*}
where $Lip(1)$ is the set of all Lipschitz functions with Lipschitz
constant $\leqslant 1$.

Let us now relate our results obtained in Theorem
\ref{thm-{Kol}-theta-tilde}   to the existing
literature, especially to the estimates  established by
\cite{bishwal2006} and  \cite{EAA2023}.
The paper  \cite{bishwal2006}    provided an
explicit upper bound for the Kolmogorov  distance for the rate  of
convergence of the distribution of
   $\widetilde{\theta}_{n}$,  but it  is not sharp. On the other
   hand, \cite{EAA2023} provided  a sharp Wasserstein bound  in
    central limit theorem for $\widetilde{\theta}_{n}$.
   Let us describe what is proved  in this
 direction:
\begin{itemize}
\item  Theorem 2.1 in \cite{bishwal2006}
shows that  there exists $C>0$ depending only  on $\theta$ such that
 \begin{eqnarray}
d_{Kol}\left(\sqrt{\frac{T}{2
\theta}}\left(\widetilde{\theta}_{n}-\theta\right),\mathcal{N} \right)\leq C \max
\left(\sqrt{\frac{\log T}{T}},\frac{T^4}{n^2\log
T}\right).\label{bishwal-eq1}\end{eqnarray}
 
\item    Theorem 3.6 in \cite{EAA2023} establishes that there exists   $C>0$ depending only  on $\theta$  such that 
   \begin{eqnarray}
d_{W}\left(\sqrt{\frac{T}{2\theta}}\left(\widetilde{\theta}_{n}-\theta\right),
\mathcal{N}\right) \leq  C
\max\left(\frac{1}{\sqrt{T}},\frac{T^2}{n^2}\right).
\label{intro-eq:bound_tilde}
\end{eqnarray}
 \end{itemize}
Note that the Wasserstein bound \eqref{intro-eq:bound_tilde} is sharper than the Kolmogorov bound \eqref{bishwal-eq1}. It is well known that if $F$ is any real-valued random variable and
$N \sim \mathscr{N}(0,1)$ is standard Gaussian, then
$$
d_{\mathrm{Kol}}(F, N) \leq 2 \sqrt{d_{\mathrm{W}}(F, N)} .
$$
(See, for example, \cite[Theorem 3.3]{CGS} or \cite[Remark
C.2.2]{NP-book}). But this relation is not enough to obtain the
estimates \eqref{intro-eq:bound_tilde}   for   the Kolmogorov metric. However this is
possible through Theorem \ref{thm-{Kol}-theta-tilde} below. Furthermore, from a statistical viewpoint,
rates of convergence under the Kolmogorov distance are usually more informative and useful
than rates of convergence  under the Wasserstein distance. For example, if
we apply Stein’s method, it is in general considerably more difficult to obtain sharp
bounds on the Kolmogorov distance than on the Wasserstein distance.

\begin{theorem}\label{thm-{Kol}-theta-tilde}
There exists $C>0$ depending only on $\theta$ such that for all
$n\geq1$,\begin{eqnarray*}
d_{Kol}\left(\sqrt{\frac{T}{2\theta}}\left(\theta-\widetilde{\theta}_{n}\right),
\mathcal{N}  \right)&\leq&
C\max\left(\frac{1}{\sqrt{T}},\frac{T^2}{n^2}\right).
\end{eqnarray*}
\end{theorem}
\begin{proof}
 We can write the approximative MCE $\widetilde{\theta}_{n}$  as
\begin{eqnarray*}
\begin{gathered}
\theta-\widetilde{\theta}_{n}
=\frac{\theta\left(f_{n}\left(X\right)-\frac{1}{2\theta}\right)}{f_{n}\left(X\right)},
\quad n \geq 1.
\end{gathered}
\end{eqnarray*}
Hence
\begin{eqnarray}
\begin{gathered}
\sqrt{\frac{T}{2\theta}}\left(\theta-\widetilde{\theta}_{n}\right)
=\frac{\frac{1}{\sigma}\sqrt{T}\left(f_{n}\left(X\right)-\mathbb{E}f_{n}\left(Z\right)\right)
+\frac{1}{\sqrt{T}}\left[\frac{T}{\sigma}\left(\mathbb{E}f_{n}\left(X\right)-\frac{1}{2\theta}\right)\right]}{\frac{1}{\rho
\sqrt{T}}\left(\sqrt{T}f_{n}\left(X\right)\right)}.
\end{gathered}\label{app-theta-tilde}
\end{eqnarray}
Thus,
$\sqrt{\frac{T}{2\theta}}\left(\theta-\widetilde{\theta}_{n}\right)$
can be written as a ratio of  the form \eqref{def-Q},   where, in
this situation, $G_T=\sqrt{T}f_{n}\left(X\right)$,  $\rho
=\frac{1}{2\theta}$, $\sigma^2=\frac{1}{2\theta^3}$, $A_T=0$ and
$a_T=\frac{T}{\sigma}\left(\mathbb{E}f_{n}\left(X\right)-\frac{1}{2\theta}\right)$.\\
Furthermore, using \eqref{var-Z},
\begin{eqnarray*}
a_T&=&\frac{T}{\sigma}\mathbb{E}\left(f_{n}\left(X\right)-f_{n}\left(Z\right)\right).
\end{eqnarray*}
Combining this with \eqref{error X-Z}, we deduce that $a_T$ is
bounded. Hence, the assumption $\mathcal{A}_2$   holds for  $A_T=0$ and
$a_T=\frac{T}{\sigma}\left(\mathbb{E}f_{n}\left(X\right)-\frac{1}{2\theta}\right)$. Setting
\[V_T=\sqrt{T}\left(f_{n}\left(Z\right)-\mathbb{E}f_{n}\left(Z\right)\right)=F_{n}\left(Z\right),\]
the process
\begin{eqnarray*}R_T=\sqrt{T}\left(G_T-\mathbb{E}G_T-V_T\right)
=T\left[\left(f_{n}\left(X\right)-f_{n}\left(Z\right)\right)-\left(\mathbb{E}f_{n}\left(X\right)-\mathbb{E}f_{n}\left(Z\right)\right)\right]
\end{eqnarray*}
   is bounded in ${\mathcal{H}}_{2}$, according to
\eqref{error X-Z}.
 Thus, by \eqref{error X-Z} and \eqref{variance-F(X)}, the assumption $\mathcal{A}_1$ is satisfied
 with
$G_T=\sqrt{T}f_{n}\left(X\right)$ and $V_T=F_{n}\left(Z\right)$. As
a consequence, combining Theorem \ref{main-thm}, \eqref{error X-Z},
\eqref{variance-F(X)}, \eqref{max(k3,k4)-estimator1} and
\eqref{app-theta-tilde},
\begin{eqnarray}
d_{Kol}\left(\sqrt{\frac{T}{2\theta}}\left(\theta-\widetilde{\theta}_{n}\right),
\mathcal{N}  \right)&\leq&
\max\left(\left|\kappa_{3}\left(\frac{V_T}{\sigma}\right)\right|,\kappa_{4}\left(\frac{V_T}{\sigma}\right)\right)
+\frac{C}{\sqrt{T}}\nonumber\\&&
+C\left| \mathbb{E}[( G_T-\mathbb{E} G_T)^2] - \sigma^2\right|+\frac{C}{T^{\frac14}}\left|\frac{1}{\rho
\sqrt{T}}\mathbb{E}G_T-1\right|\nonumber\\
&=&
\max\left(\left|\kappa_{3}\left(\frac{F_{n}\left(Z\right)}{\sigma}\right)\right|,\left|\kappa_{4}\left(\frac{F_{n}\left(Z\right)}{\sigma}\right)\right|\right)
+\frac{C}{\sqrt{T}}\nonumber\\&&
+C\left| \mathbb{E}[ F_{n}\left(X\right)^2] - \sigma^2\right| +\frac{2\theta C}{T^{\frac14}}\left|\mathbb{E}\left(f_{n}\left(X\right)-f_{n}\left(Z\right)\right)\right|  \nonumber\\
&\leq& C\left(\Delta_n^2+\frac{1}{\sqrt{T}}\right)
\label{{Kol}-proof-tilde}\\
&\leq& C\max\left(\Delta_n^2,\frac{1}{\sqrt{T}}\right).\nonumber
\end{eqnarray}
Therefore the proof is done.
\end{proof}

\subsection{Approximate  maximum likelihood estimators}\label{AMLE-sect}
In this section  we study
Kolmogorov bounds in CLT of 
approximate maximum likelihood
estimators (AMLEs) of   the drift parameter $\theta$ for the
Ornstein-Uhlenbeck process \eqref{INTRO-OU},  observed
at high frequency.
The maximum likelihood estimator  for $\theta$ based on continuous
observations of the process $X$ given by \eqref{INTRO-OU}, is defined by
\begin{eqnarray}
 \check{\theta}_{T}=-\frac{\int_{0}^{T} X_{s} \mathrm{~d} X_{s}}{\int_{0}^{T}
X_{s}^{2} \mathrm{~d} s}, \quad T\geq 0. \label{MLE-cont}
\end{eqnarray}
On the other hand, using the It\^{o} formula, we can write $\int_{0}^{T} X_{s}
\mathrm{~d} X_{s}=\frac12\left(X_T^2-T\right)$. Hence, the estimator
$\check{\theta}_{T}$ given in \eqref{MLE-cont}, can be rewritten as
follows
\begin{eqnarray}
 \check{\theta}_{T}=-\frac{\frac12\left(X_T^2-T\right)}{\int_{0}^{T}
X_{s}^{2} \mathrm{~d} s}, \quad T\geq 0. \label{MLE-cont-2nd-form}
\end{eqnarray}
Here we are concerned with the approximate  maximum likelihood estimators
\begin{eqnarray}\widehat{\theta}_{n}:= -\frac{\sum_{i=1}^{n}
X_{t_{i-1}}\left(X_{t_{i}}-X_{t_{i-1}}\right)}{\Delta_{n}
\sum_{i=1}^{n} X_{t_{i-1}}^{2}},\label{AMLE!}\end{eqnarray}
and
\begin{eqnarray}
\bar{\theta}_{n}=\frac{\frac12\left(X_T^2-T\right)}{\Delta_n\sum_{i=1}^n
X^2_{t_{i-1}}},\label{AMLE2}
\end{eqnarray}
which are  discrete versions of \eqref{MLE-cont} and \eqref{MLE-cont-2nd-form}, respectively. \\
Rates of convergence in CLT of the AMLEs $\widehat{\theta}_{n}$ and   $\bar{\theta}_{n}$
under the Kolmogorov   distance  have been studied as
follows: 
\begin{itemize}
\item Theorem 2.3 in \cite{BB} shows that there exist $ C>0$ depending only on $\theta$ such that
\begin{eqnarray}d_{Kol}\left(\sqrt{\frac{T}{2
\theta}}\left(\widehat{\theta}_{n}-\theta\right) ,\mathcal{N} \right)\leq C \max
\left(\sqrt{\frac{\log T}{T}},\frac{T^2}{n\log
T}\right).\label{bishwal-eq2}\end{eqnarray}
\item Theorem 3.1 in \cite{BB} proves that there exist $ C>0$ depending only on $\theta$ such that
 \begin{eqnarray}d_{Kol}\left(\sqrt{\frac{T}{2
\theta}}\left(\bar{\theta}_{n}-\theta\right) , \mathcal{N}\right)\leq C\max \left(\sqrt{\frac{\log
T}{T}},\frac{T^4}{n^2\log T}\right).\label{bishwal-eq3}\end{eqnarray}
\end{itemize}
Moreover, according to Theorem 4.1 in \cite{EAA2023}, a Wasserstein bound in the central limit
theorem for $\widehat{\theta}_{n}$ has been provided as follows:  
   There exists a constant $C>0$ such that, for all $n\geq1$,
\begin{eqnarray}
 d_W\left(\sqrt{\frac{T}{2\theta}}\left(\widehat{\theta}_{n}-\theta\right),\mathcal{N}\right)
 &\leq&C\max \left(\frac{1}{\sqrt{T}},\sqrt{\frac{T^3}{n^2}}\right).\label{rate2-hat}
\end{eqnarray}

In the following remark we compare these estimates with   the forthcoming  results of
Theorem \ref{{Kol}-theta-hat} and Theorem \ref{{Kol}-theta-bar}.

\begin{remark} The Kolmogorov bounds appearing in  Theorem \ref{{Kol}-theta-hat} and Theorem \ref{{Kol}-theta-bar} show that we have
  improved the bounds on the error of normal approximation for $\widehat{\theta}_{n}$ and
  $\bar{\theta}_{n}$. Clearly, Theorem \ref{{Kol}-theta-hat} and Theorem \ref{{Kol}-theta-bar} prove that the estimators $\widehat{\theta}_{n}$ and
  $\bar{\theta}_{n}$ have the same rate of convergence in the CLT under the Kolmogorov metric, and which  is sharper than the bounds in
\eqref{bishwal-eq2}, \eqref{bishwal-eq3} and \eqref{rate2-hat}.
  \end{remark}

\begin{theorem}\label{{Kol}-theta-hat}
Suppose that $\frac{T^3}{n^2}\rightarrow 0$. Then, there exists $C>0$ depending
only on $\theta$ such that for all $n\geq1$,\begin{eqnarray*}
d_{Kol}\left(\sqrt{\frac{T}{2\theta}}\left(\theta-\widehat{\theta}_{n}\right)
,\mathcal{N}   \right)&\leq&
C\max\left(\frac{1}{\sqrt{T}},\sqrt{\frac{T}{n}},\sqrt{\frac{T^3}{n^2}}\right).
\end{eqnarray*}
\end{theorem}
\begin{proof}
In this case, $\rho =\frac{1}{2\theta}$ and
$\sigma^2=\frac{1}{2\theta^3}$. The AMLE
$\widehat{\theta}_{n}$ can be written as
\begin{eqnarray*}
\begin{gathered}
\theta-\widehat{\theta}_{n} = \frac{\sum_{i=1}^{n}
X_{t_{i-1}}\left(X_{t_{i}}-X_{t_{i-1}}\right)+\theta\Delta_{n}
\sum_{i=1}^{n} X_{t_{i-1}}^{2}}{\Delta_{n} \sum_{i=1}^{n}
X_{t_{i-1}}^{2}},\quad n\geq1.
\end{gathered}
\end{eqnarray*}
Hence
\begin{eqnarray}
\sqrt{\frac{T}{2\theta}}\left(\theta-\widehat{\theta}_{n}\right)
&=&\frac{\sqrt{\frac{2\theta}{T}}\left(\sum_{i=1}^{n}
X_{t_{i-1}}\left(X_{t_{i}}-X_{t_{i-1}}\right)+\theta\Delta_{n}
\sum_{i=1}^{n} X_{t_{i-1}}^{2}\right)}{\frac{1}{\rho
\sqrt{T}}\left(\sqrt{T}f_{n}\left(X\right)\right)}\\&=&\frac{\frac{1}{\sigma}(G_T-\mathbb{E}G_T)+\frac{1}{\sqrt{T}}A_T}{\frac{1}{\rho
\sqrt{T}}G_T}, \label{app-theta-hat}
\end{eqnarray}
where, in this case, $G_T=\sqrt{T}f_{n}\left(X\right)$, $\rho
=\frac{1}{2\theta}$, $\sigma^2=\frac{1}{2\theta^3}$,  $a_T=0$   and
\[\frac{1}{\sigma}(G_T-\mathbb{E}G_T)+\frac{1}{\sqrt{T}}A_T=\sqrt{\frac{2\theta}{T}}\left(\sum_{i=1}^{n}
X_{t_{i-1}}\left(X_{t_{i}}-X_{t_{i-1}}\right)+\theta\Delta_{n}
\sum_{i=1}^{n} X_{t_{i-1}}^{2}\right).\]
 Moreover,
\[V_T=\sqrt{T}\left(f_{n}\left(Z\right)-\mathbb{E}f_{n}\left(Z\right)\right)\]
and
\begin{eqnarray*}R_T=\sqrt{T}\left(G_T-\mathbb{E}G_T-V_T\right)
=T\left[\left(f_{n}\left(X\right)-f_{n}\left(Z\right)\right)-\left(\mathbb{E}f_{n}\left(X\right)-\mathbb{E}f_{n}\left(Z\right)\right)\right].
\end{eqnarray*}
Since the processes $G_T$, $ V_T$ and $R_T$ are exactly the same as in Theorem \eqref{thm-{Kol}-theta-tilde}, then the assumption
$\mathcal{A}_1$   holds for those processes. Now, it remains to prove that the assumption
$\mathcal{A}_2$ holds.\\
 Using \eqref{decomp-Num in Gamma+S},
\begin{eqnarray}
A_T&=&\sqrt{2\theta}\left( \sum_{i=1}^{n}
X_{t_{i-1}}\left(X_{t_{i}}-X_{t_{i-1}}\right)+\theta\Delta_{n}
\sum_{i=1}^{n} \mathbb{E}X_{t_{i-1}}^{2}  \right)\nonumber\\
&=&\sqrt{2\theta}\left( \frac{e^{-\theta
\Delta_{n}}-1}{\Delta_{n}}S_{n}+ \Lambda_{n}+\theta\mathbb{E}S_n
\right)\nonumber\\
&=&\sqrt{2\theta}\left( \left[\left(\frac{e^{-\theta
\Delta_{n}}-1}{\Delta_{n}}+\theta\right)\left(S_{n}-\mathbb{E}S_{n}\right)\right]+
\left[\left(\frac{e^{-\theta
\Delta_{n}}-1}{\Delta_{n}}+\theta\right) \mathbb{E}S_{n}\right]
+\left[\Lambda_{n}-\theta
\left(S_{n}-\mathbb{E}S_{n}\right)\right]\right)\nonumber\\
&=:&\sqrt{2\theta}\left(A_{1,T}+A_{2,T}+A_{3,T}\right).\label{decomp.-AT}
\end{eqnarray}
On the other hand,
\begin{eqnarray*}A_{1,T}=\left(\frac{\theta^2}{2}\Delta_{n}+o(\Delta_{n})\right)\sqrt{T}F_n(X).
\end{eqnarray*}
Hence, by \eqref{variance-F(X)},
\begin{equation*}
\left\|A_{1,T}\right\|_{L^2(\Omega)}\leq
C\sqrt{n\Delta_n^3}, 
\end{equation*}
so, \begin{equation}
\frac{1}{\sqrt{T}}\left\|A_{1,T}\right\|_{L^2(\Omega)}\leq
C \Delta_n.\label{norm-A1}
\end{equation}
Since
\begin{eqnarray*}A_{2,T}=\left(\frac{\theta^2}{2}\Delta_{n}+o(\Delta_{n})\right)T\mathbb{E}f_n(X),
\end{eqnarray*}
then, according to  \eqref{error X-Z},
\begin{equation}
\frac{1}{\sqrt{T}}\left\|A_{2,T}\right\|_{L^2(\Omega)}\leq C\frac{n\Delta_n^2}{\sqrt{T}}=C\sqrt{n\Delta_n^3} .\label{norm-A2}
\end{equation}
Moreover, we can write
\[A_{3,T}=\Lambda_{n}-\theta
\sqrt{T}F_{n}(X). \] So, using \eqref{cv-Lambda} and
\eqref{variance-F(X)}, we can write
\begin{eqnarray}\mathbb{E}\left(A_{3,T}^2\right)&=&\mathbb{E}\left(\Lambda_{n}^2\right) +\theta^2T\mathbb{E}\left(F_{n}(X)^2\right)-
2\theta \sqrt{T}\mathbb{E}\left(\Lambda_{n}F_{n}(X)\right)\nonumber\\
&\leq&\left|\mathbb{E}\left(\Lambda_{n}^2\right)-\frac{T}{2\theta}\right|
+\theta^2T\left|\mathbb{E}\left(F_{n}(X)^2\right)-\frac{1}{2\theta^3}\right|+\left|\frac{T}{\theta}
- 2\theta
\sqrt{T}\mathbb{E}\left(\Lambda_{n}F_{n}(X)\right)\right|\nonumber\\
&\leq&C\left(n\Delta_n^2+1\right) +
C\left(n\Delta_n^3+1\right)+\left|\frac{T}{\theta} - 2\theta
\sqrt{T}\mathbb{E}\left(\Lambda_{n}F_{n}(X)\right)\right|\nonumber
\\
&\leq&C\left(n\Delta_n^2+1\right)+\left|\frac{T}{\theta} - 2\theta
\sqrt{T}\mathbb{E}\left(\Lambda_{n}F_{n}(X)\right)\right|.\label{ineq.A3}
\end{eqnarray}Further, by $\mathbb{E}\Lambda_{n}=0$,
\begin{eqnarray*}
\frac{T}{\theta} - 2\theta
\sqrt{T}\mathbb{E}\left(\Lambda_{n}F_{n}(X)\right)&=&\frac{T}{\theta}
- 2\theta \mathbb{E}\left(\Lambda_{n}S_{n}(X)\right)\\
&=&\frac{T}{\theta} - 2\theta\Delta_n \mathbb{E}\left[
\sum_{i=1}^{n}
e^{-\theta(t_{i}+t_{i-1})}\zeta_{t_{i-1}}\left(\zeta_{t_{i}}-\zeta_{t_{i-1}}\right)
\sum_{j=1}^{n} e^{-2\theta t_{j-1}}\zeta_{t_{j-1}}^2\right]\\
&=&\frac{T}{\theta} - 2\theta\Delta_n \sum_{i,j=1,i<j}^{n}
e^{-\theta(t_{i}+t_{i-1})} e^{-2\theta
t_{j-1}}\mathbb{E}\left[\zeta_{t_{i-1}}\left(\zeta_{t_{i}}-\zeta_{t_{i-1}}\right)
\zeta_{t_{j-1}}^2\right].
\end{eqnarray*}
Now, using the Wick formula, we get
\begin{eqnarray*}
\frac{T}{\theta} - 2\theta
\sqrt{T}\mathbb{E}\left(\Lambda_{n}F_{n}(X)\right)&=&
 \frac{T}{\theta} - 4\theta\Delta_n \sum_{i,j=1,i<j}^{n}
e^{-\theta(t_{i}+t_{i-1})} e^{-2\theta
t_{j-1}}\mathbb{E}\left[\zeta_{t_{i-1}}^2\right]\mathbb{E}\left[\left(\zeta_{t_{i}}-\zeta_{t_{i-1}}\right)^2\right]\\
&=&
 \frac{T}{\theta} - 4\theta\Delta_n \sum_{j=2}^{n}\sum_{i=1}^{j-1}
e^{-\theta(t_{i}+t_{i-1})} e^{-2\theta t_{j-1}}
\left[\frac{e^{2\theta
t_{i-1}}-1}{2\theta}\right]\left[\frac{e^{2\theta t_{i}}-e^{2\theta
t_{i-1}}}{2\theta}\right]\\
&=&
 \frac{T}{\theta} -\frac{\Delta_n}{\theta} \sum_{j=2}^{n} e^{-2\theta t_{j-1}}\sum_{i=1}^{j-1}
e^{\theta(t_{i}+t_{i-1})} \left[ 1-e^{-2\theta
t_{i-1}}\right]\left[1-e^{-2\theta \Delta_n}\right]\\
&=&
 \left(\frac{T}{\theta} -\frac{\Delta_n}{\theta} \left[1-e^{-2\theta \Delta_n}\right]
 \sum_{j=2}^{n} e^{-2\theta t_{j-1}}\sum_{i=1}^{j-1}
e^{\theta(t_{i}+t_{i-1})}\right)  \\
&&+\left(\frac{\Delta_n}{\theta}\left[1-e^{-2\theta \Delta_n}\right]
\sum_{j=2}^{n} e^{-2\theta t_{j-1}}\sum_{i=1}^{j-1}
e^{\theta\Delta_n}\right)\\
&=:&a_{n,T}+b_{n,T},
\end{eqnarray*}
where
\begin{eqnarray*} a_{n,T}&=&  \frac{T}{\theta} -\frac{\Delta_n}{\theta} \left[1-e^{-2\theta \Delta_n}\right]
 \sum_{j=2}^{n} e^{-2\theta \Delta_n(j-1)}\sum_{i=1}^{j-1}e^{\theta\Delta_n}
e^{2\theta\Delta_n(i-1)}\\
&=&  \frac{T}{\theta} -\frac{\Delta_n}{\theta} \left[1-e^{-2\theta
\Delta_n}\right]e^{\theta\Delta_n}
 \sum_{j=2}^{n} e^{-2\theta \Delta_n(j-1)}\frac{e^{2\theta\Delta_n(j-1)}-1}{e^{2\theta\Delta_n}-1}\\
&=&  \frac{T}{\theta} -\frac{\Delta_n}{\theta}  e^{-\theta\Delta_n}
 \sum_{j=2}^{n}\left(1-e^{-2\theta\Delta_n(j-1)}\right)\\
&=&  \frac{T}{\theta} -\frac{\Delta_n}{\theta}  e^{-\theta\Delta_n}
  \left[n-1-e^{-2\theta\Delta_n}\frac{\left(1-e^{-2\theta\Delta_n(n-1)}\right)}{1-e^{-2\theta\Delta_n}}\right]\\
&=&  \frac{T}{\theta}\left(1-e^{-\theta\Delta_n}\right)
+\frac{\Delta_n}{\theta} e^{-\theta\Delta_n}
 +\frac{\Delta_n}{\theta}e^{-3\theta\Delta_n}\frac{\left(1-e^{-2\theta\Delta_n(n-1)}\right)}{1-e^{-2\theta\Delta_n}}.
\end{eqnarray*}
Thus, by the fact that
$\displaystyle{\frac{1-e^{-2\theta\Delta_n}}{\Delta_n}}\rightarrow2\theta$
and
$\displaystyle{\frac{1-e^{-\theta\Delta_n}}{\Delta_n}}\rightarrow
\theta$ as $n\rightarrow\infty$,
\begin{eqnarray} \left|a_{n,T}\right|\leq
C\left(n\Delta_n^2+\Delta_n+1\right).\label{upper-a}
\end{eqnarray}
Similarly, using  $xe^{-2\theta x}\leq C e^{-\theta x}$
 for all $x\geq0$,
\begin{eqnarray}
b_{n,T}&=& \frac{\Delta_n}{\theta}\left[1-e^{-2\theta
\Delta_n}\right] \sum_{j=2}^{n} e^{-2\theta t_{j-1}} (j-1)
e^{\theta\Delta_n}\nonumber\\
&\leq& C\frac{\Delta_ne^{\theta\Delta_n}}{\theta}\left[
1-e^{-2\theta \Delta_n} \right]  \sum_{j=2}^{n}  e^{-\theta
(j-1)\Delta_n}\nonumber\\
&=& C\frac{\Delta_ne^{\theta\Delta_n}}{\theta}\left[ 1-e^{-2\theta
\Delta_n} \right]   \frac{1-e^{-\theta (n-1)\Delta_n}}{1-e^{-\theta
\Delta_n}}\nonumber\\
&\leq& C\frac{\Delta_ne^{\theta\Delta_n}}{\theta}\left[ 1+e^{-\theta
\Delta_n} \right]\nonumber\\
&\leq& C \Delta_n.\label{upper-b}
\end{eqnarray}
Hence, it follows from \eqref{upper-a} and \eqref{upper-b} that
\begin{eqnarray*}\left|\frac{T}{\theta} - 2\theta
\sqrt{T}\mathbb{E}\left(\Lambda_{n}F_{n}(X)\right)\right|\leq
C\left(n\Delta_n^2+\Delta_n+1\right).
\end{eqnarray*}
This estimate and \eqref{ineq.A3} imply
\begin{eqnarray}\frac{1}{\sqrt{T}}\left\|A_{3,T}\right\|_{L^2(\Omega)}
\leq\frac{C}{\sqrt{T}}\sqrt{n\Delta_n^2+\Delta_n+1}
\leq C\left(\sqrt{\Delta_n}+\frac{1}{\sqrt{n}}+\frac{1}{\sqrt{n\Delta_n}}\right).\label{norm-A3}\end{eqnarray}
Combining \eqref{decomp.-AT}, \eqref{norm-A1}, \eqref{norm-A2},
\eqref{norm-A3} and $n\Delta_n^3=\frac{T^3}{n^2}\rightarrow 0$, we deduce that $\frac{1}{\sqrt{T}}A_T\rightarrow0$  in $L^2(\Omega)$,  and consequently, the assumption
$\mathcal{A}_2$   holds, and moreover, \[\frac{1}{\sqrt{T}}\left\|A_{T}\right\|_{L^2(\Omega)}\leq C\max\left(\frac{1}{\sqrt{T}},\sqrt{\frac{T}{n}},\sqrt{\frac{T^3}{n^2}}\right).\] Hence,
the conditions of Theorem \ref{main-thm} are satisfied for
\[G_T=\sqrt{T}f_{n}\left(X\right),\quad \frac{1}{\sigma}(G_T-\mathbb{E}G_T)+\frac{1}{\sqrt{T}}A_T=\sqrt{\frac{2\theta}{T}}\left(\sum_{i=1}^{n}
X_{t_{i-1}}\left(X_{t_{i}}-X_{t_{i-1}}\right)+\theta\Delta_{n}
\sum_{i=1}^{n} X_{t_{i-1}}^{2}\right),\] $a_T=0$ and
$V_T=F_{n}\left(Z\right)$. Therefore, using a similar argument as in \eqref{{Kol}-proof-tilde},
the proof is complete.
\end{proof}

\begin{theorem}\label{{Kol}-theta-bar}
  There exists $C>0$ depending only
on $\theta$ such that for all $n\geq1$,\begin{eqnarray*}
d_{Kol}\left(\sqrt{\frac{T}{2\theta}}\left(\theta-\bar{\theta}_{n}\right)
,\mathcal{N}  \right)&\leq&
C\max\left(\frac{1}{\sqrt{T}},\frac{T^2}{n^2}\right).
\end{eqnarray*}
\end{theorem}
\begin{proof}
It follows from \eqref{AMLE2} that
\begin{eqnarray}
\begin{gathered}
\sqrt{\frac{T}{2\theta}}\left(\theta-\bar{\theta}_{n}\right)
=\frac{\sqrt{\frac{2\theta}{T}}\left(\frac12\left(X_T^2-T\right)+\theta\Delta_{n}
\sum_{i=1}^{n}
\mathbb{E}X_{t_{i-1}}^{2}\right)}{\frac{2\theta}{\sqrt{T}}\left(\sqrt{T}f_{n}\left(X\right)\right)}
=\frac{\frac{1}{\sigma}(G_T-\mathbb{E}G_T)+\frac{1}{\sqrt{T}}A_T}{\frac{1}{\rho
\sqrt{T}}G_T},
\end{gathered}\label{app-theta-hat}
\end{eqnarray}
where, in this situation,  $G_T=\sqrt{T}f_{n}\left(X\right)$, $\rho
=\frac{1}{2\theta}$, $\sigma^2=\frac{1}{2\theta^3}$,  $a_T=0$ and
\[\frac{1}{\sigma}(G_T-\mathbb{E}G_T)+\frac{1}{\sqrt{T}}A_T=\sqrt{\frac{2\theta}{T}}\left(\frac12\left(X_T^2-T\right)+\theta\Delta_{n}
\sum_{i=1}^{n} \mathbb{E}X_{t_{i-1}}^{2}\right).\]
 Moreover,
\[V_T=\sqrt{T}\left(f_{n}\left(Z\right)-\mathbb{E}f_{n}\left(Z\right)\right)\]
and
\begin{eqnarray*}R_T=\sqrt{T}\left(G_T-\mathbb{E}G_T-V_T\right)
=T\left[\left(f_{n}\left(X\right)-f_{n}\left(Z\right)\right)-\left(\mathbb{E}f_{n}\left(X\right)-\mathbb{E}f_{n}\left(Z\right)\right)\right].
\end{eqnarray*}
In this case, we also notice that  the processes $G_T$, $ V_T$ and $R_T$ are exactly the same as in Theorem \eqref{thm-{Kol}-theta-tilde}, so, the assumption
$\mathcal{A}_1$   holds for those processes.  Now, it remains to prove that the assumption
$\mathcal{A}_2$ holds.\\
It follows from \eqref{X-Z} that $\theta\Delta_{n}
\sum_{i=1}^{n} \mathbb{E}Z_{t_{i-1}}^{2}=\frac{T}{2}$, so, we can
write
\begin{eqnarray*}
A_T&=&\sqrt{T}\left(\frac{1}{\sigma}(G_T-\mathbb{E}G_T)+\frac{1}{\sqrt{T}}A_T-\frac{1}{\sigma}(G_T-\mathbb{E}G_T)\right)
=\sqrt{2\theta}\left( \frac12\left(X_T^2-T\right)+\theta\Delta_{n}
\sum_{i=1}^{n} \mathbb{E}X_{t_{i-1}}^{2}  \right)\\
&=&\sqrt{2\theta}\left( \frac12X_T^2- \theta\Delta_{n}
\sum_{i=1}^{n}
\left[\mathbb{E}X_{t_{i-1}}^{2}-\mathbb{E}Z_{t_{i-1}}^{2}  \right]
\right).
\end{eqnarray*}
Moreover, since $X_T$ is Gaussian, it is easy to see
 \begin{equation}
\mathbb{E}X_{T}^4=3\left(\mathbb{E}X_{T}^2\right)^2=3\left(\mathbb{E}Z_{T}^2-2e^{-\theta
T}\mathbb{E}\left(Z_0Z_T\right)+e^{-2\theta
T}\mathbb{E}Z_{0}^2\right)^2\leq C,
\end{equation}
and
 \begin{eqnarray*}
\left|\Delta_{n} \sum_{i=1}^{n}
\left[\mathbb{E}X_{t_{i-1}}^{2}-\mathbb{E}Z_{t_{i-1}}^{2}
\right]\right|&=&\left|\Delta_{n} \sum_{i=1}^{n} \left[-2e^{-\theta
t_{i-1}}\mathbb{E}\left(Z_0Z_{t_{i-1}}\right)+e^{-2\theta
t_{i-1}}\mathbb{E}Z_{0}^2 \right]\right|\\ &\leq&  C \Delta_{n}
\sum_{i=1}^{n}  e^{-\theta t_{i-1}}\\ &=& C \Delta_{n}
\frac{1-e^{-\theta n\Delta_n}}{1-e^{-\theta \Delta_n}}
\\ &=&C,
\end{eqnarray*}
since
$\displaystyle{\frac{1-e^{-\theta\Delta_n}}{\Delta_n}}\rightarrow
\theta$ as $n\rightarrow\infty$.\\
Consequently, $A_T$ is bounded in $L^2(\Omega)$,  and then, the
assumption $\mathcal{A}_2$  holds. Also, the conditions of Theorem
\ref{main-thm} are satisfied for $G_T=\sqrt{T}f_{n}\left(X\right)$,
$V_T=F_{n}\left(Z\right)$, $a_T=0$ and
\[\frac{1}{\sigma}(G_T-\mathbb{E}G_T)+\frac{1}{\sqrt{T}}A_T=\sqrt{\frac{2\theta}{T}}\left(\frac12\left(X_T^2-T\right)+\theta\Delta_{n}
\sum_{i=1}^{n} \mathbb{E}X_{t_{i-1}}^{2}\right)\]. Therefore, using \eqref{{Kol}-proof-tilde},
the proof is done. \end{proof}

\end{document}